\documentclass[journal,draftcls,onecolumn]{IEEEtran}

\usepackage{url}
\usepackage{ifpdf}
\ifpdf
\usepackage[pdftex]{graphicx}
\else
\usepackage[dvips]{graphicx}
\fi

\DeclareGraphicsExtensions{.eps}
\usepackage{epstopdf}
\usepackage{graphics} 
\usepackage{epsfig} 
\DeclareGraphicsExtensions{.eps}
\usepackage{tikz}
\usepackage{bm}
\usepackage{verbatim}

\usepackage{algorithm}
\usepackage{algpseudocode}
\usepackage{amsmath} 
\usepackage{amssymb}  
\usepackage{threeparttable}

\newcommand{\E}{\mathbb{E}}
\DeclareMathOperator{\Var}{\text{Var}}
\DeclareMathOperator{\Cov}{\text{Cov}}
\usepackage{amsthm}
\newtheorem{theorem}{Theorem}
\newtheorem{lemma}{Lemma}
\newtheorem{remark}{Remark}

\newcommand{\bd}{\mathbf}
\begin{document}

\title{Linear Estimation of Treatment Effects in Demand Response: An Experimental Design Approach}

\author{Pan~Li,~\IEEEmembership{Student~Member,~IEEE,}~and~Baosen Zhang,~\IEEEmembership{Member,~IEEE}
	\thanks{This work was supported by in part by NSF grant CNS-1544160 and the University of Washington Clean Energy Institute.}
	\thanks{The authors are with the Department of Electrical Engineering, University of Washington, Seattle, WA 98195, USA (e-mail: \{pli69,zhangbao\}@uw.edu).}}


\maketitle

\begin{abstract}
Demand response aims to stimulate electricity consumers to modify their loads at critical time periods.
In this paper, we consider signals in demand response programs as a binary treatment to the customers and estimate the average treatment effect, which is the average change in consumption under the demand response signals. More specifically, we propose to estimate this effect by linear regression models and derive several estimators based on the different models. From both synthetic and real data, we show that including more information about the customers does not always improve estimation accuracy: the interaction between the side information and the demand response signal must be carefully modeled. In addition, we compare the traditional linear regression model with the modified covariate method which models the interaction between treatment effect and covariates. We analyze the variances of these estimators and discuss different cases where each respective estimator works the best. The purpose of these comparisons is not to claim the superiority of the different methods, rather we aim to provide practical guidance on the most suitable estimator to use under different settings.
Our results are validated using data collected by Pecan Street and EnergyPlus.
\end{abstract}
%

\section*{Nomenclature}
\addcontentsline{toc}{section}{Nomenclature}
\begin{IEEEdescription}[\IEEEsetlabelwidth{$\mathbf{v}_i = (T_i - p) \mathbf{x}_i$}]
	\item[$N$] Total number of samples.
	\item[$i$] Index of samples, $i = 1,2,..., N$.
	\item[$Y_i(\mathbf{Y})$] Energy consumption (vector form).
	\item[$T_i$] Binary DR signal~(treatment signal) for sample $i$.
	\item[$\mathbf{x}_i (\mathbf{x})$] Covariates of sample $i$ (matrix form).
	\item[$d$] Dimension of the covariate.
	\item[$g(\mathbf{x}_i) \mbox{ or } g_i$] Treatment effect of DR signal on sample $i$.
	\item[$\bar{g}=\frac{1}{N} \sum_i g_i$] Average treatment effect (ATE): the average change in consumption because of demand response signals.
	\item[$f(\mathbf{x}_i) \mbox{ or } f_i$] Baseline consumption~(without DR) of sample $i$.
	\item[$p$] Treatment assignment probability.
	\item[$\bm{\beta}$] Weights of a linear regression.
	\item[$\bd w$] Regressor of a linear regression.
	\item[$\bm W$] Regression matrix in a linear regression.
	\item [$\epsilon_i$] Noise in sample $i$.
	\item[$Z_i = T_i - p$] Centered treatment signal for sample $i$.
	\item[$\mathbf{v}_i = (T_i - p) \mathbf{x}_i$] Modified covariate for sample $i$.
	\item[$\mu = \frac{\sum_i x_i}{N}$] Empirical mean of covariates.
\end{IEEEdescription}

\section{Introduction}
\label{sec:intro}
One of the most interesting changes taking place in the electrical grid is that demand is no longer treated as fixed loads. Instead, operators are starting to explore \emph{demand response (DR)}, an umbrella term capturing mechanisms that modify the electricity consumption of consumers to balance the changes in generation.
Demand response has received significant attention from the community in recent years~(e.g. see \cite{Siano2014} and the references within). In typical implementation of DR programs, customers receive a signal such as a change in price or a message requesting modifications in electricity usage. An effective DR program improves the efficiency and sustainability of power systems by allowing utilities and operators to leverage flexibility in the load rather than relying entirely on the control of generation~\cite{DR0,DR2,Dai,DR3,DR1,LiEtAl2011}.
%

Most of the work in the literature have viewed the DR problem from optimization or market design perspectives. For example, authors in \cite{LiEtAl2011,QianEtAl2013} considered how to optimize the social welfare; and authors in \cite{SaadEtAl2012, LiEtAl2017} have considered how to create an efficient market for demand response. In all these setups, customers' responses to demand response signals are either assumed to be known to the operators, or at least known to themselves. \emph{However, in practice, accurately estimating how customers respond to signals is a crucial step to the design and evaluation of DR programs}. As a motivating example used throughout this paper, consider a building that participates in a demand response program. The building manager may receive a signal and take a set of actions, but the consumption of building depends on a multitude of (exogenous) variables, including external temperature, occupancy, etc. Therefore estimating the change in consumption because of a DR signal is not a trivial problem.

The lack of information about users' behaviors is a fundamental difficulty in judging the impact of demand response programs and raises questions about their effectiveness~\cite{WooEtAl2006}. Fortunately, new sources of data such as smartmeters and other sensors provide the possibility to understand customer responses better. In this paper, we take a statistical view on how these data should be used to estimate responses to DR signal. \emph{In particular, we provide guidance on various popular estimation methodologies and show that more data, if used naively, may actually degrade estimator performances.}


We adopt an \emph{experimental design approach} and focus on the problem of estimating customer responses to DR signals from observational data. The DR signal is perceived as a \emph{treatment} and is applied to some of the users~\cite{WLS}. The quantity of interest in this model is the \emph{average treatment effect (ATE)}, representing the average response of the customers to the treatment. The ATE quantifies the impact~(or effectiveness) of the DR program. The fundamental question we investigate here is to determine if the ATE is statistically significant, and if so, estimating its value. We focus on linear regression models that take observation data as inputs and energy consumption as outputs. We choose to emphasize linear models in our analysis because of their ubiquity in both theory and practice, and because they can be applied with relatively small amount of data~\cite{causalreview1,Zub_2012}.

The data used in these regressions consist of two types of variables: a \emph{binary treatment variable} indicating whether a DR signal has been received,\footnote{We can extend our results to continuous values of DR signals, although in practice most signals only consists of a few discrete values. For example, we can use 1 to represent receiving a text (0 for not receiving it) or 0 for normal price and $1$ for high prices.} and the rest of the variables are referred collectively to as \emph{covariates}. These covariates represent the rest of the observations that are available, for example, temperature, HVAC status, other appliance information, etc. Our problem focuses on \emph{inference}, where we only want to find the relationship between the DR signal and energy consumption, while the effect of the covariates are not of interest. An important challenge in this inference problem is that the impact of covariates on consumption is usually much larger than the impact of the DR signal.


We present three linear regression estimators of the ATE: 1) simple linear regression~(difference-in-mean estimator~\cite{causalreview1}), 2) multiple linear regression and 3) modified covariate method~\cite{covariatenew}. The simple linear regression estimator only uses information about whether a DR signal is sent and ignores other information, whereas  multiple linear regression incorporates the covariates by assuming a linear relationship between them and energy consumption. At first glance, the latter seems to be an natural improvement over the former since more data are used. However, as we show in the paper, multiple linear regression actually preforms worse in some settings.

The reason that multiple linear regression can perform badly is because of two reasons. The first is that the underlying function relating DR signals and covariates to consumption is almost never linear. The second is that in practice the number of DR signals received by any one user is usually small. By naively fitting a linear model, we may inadvertently introduce more noise than information, especially because the effect of DR is observed only for handful of times. To leverage the covariates information even when the underlying model is not linear, we introduce the \emph{modified covariates method}, striking a balance between simple and multiple linear regression. We validate our results using both synthetic data, building data, and data from Pecan Street~\cite{pecan}. The main contributions of this paper are:
\begin{enumerate}
  \item We study the problem of estimating the average response of users to a DR signal by considering three different estimators: simple linear regression, multiple linear regression, and modified covariate method. We show which method is the best to use under different information and signaling frequency settings. In particular, we show that using more data does not necessarily improve the estimation of the average response of DR signals if the underlying consumption model is not carefully considered.
  \item We provide theoretical guidelines on the results that are validated using both synthetic and real data from different sources. 
\end{enumerate}

The rest of the paper is organized as follows. Section \ref{sec:motiv} presents the motivation of this paper using a large building as an example. Section \ref{sec:setup} introduces the linear model and assumptions throughout this paper. Section \ref{sec:lr} presents several different estimators based on various forms of linear regression, followed by the performance analysis based in nominal variance in Section \ref{sec:varanalysis}. Section \ref{sec:sim} details the case study on the performance of these estimators. Section \ref{sec:conclusion} concludes the paper.

\section{Motivation}\label{sec:motiv}
The fact that adding more data does not always improve estimation of the average treatment effect is not new. This can be seen as the difference between prediction and inference \cite{Freedman2009}. Adding more data will almost always improve prediction. But in our context, we are trying to estimate \emph{the effect of a single variable, the DR signal, on the output}. Naively using more knowledge about customers may actually ``drown out'' the relationship between the DR signal and customer consumption. A message from this paper is that the interactions between covariates and the treatment variable need to be carefully modeled to correctly leverage additional information contained in the covariates.

The need to provide inference also limits us in our choice of algorithms. Popular machine learning tools such as neural networks and regression trees often improve prediction, but they are not easy to interpret and use in estimation~(see \cite{Li17} and the references within). Hence we study three types of linear models in this paper: simple linear regression (SLR), multiple linear regression (MLR), and modified covariate method (MCM). SLR has only one regressor which is the indicator of whether a DR signal is received, MLR has multiple regressors including the DR signal and the covariates that may affect consumption level. MCM is a multiple linear regression model with modified covariates discussed in detail in Section \ref{sec:lr}.

As a motivating example, we consider a large multifamily residential building with heating and cooling devices that participates in a demand response program. There are 8 apartments with central corridor on each floor, and office on first floor. An illustration of the building is shown in Fig. \ref{building}. Here we consider time as discrete intervals~(on the scale of hours), and during any interval the building may receive a DR signal. This model of the building is constructed in the EnergyPlus software~\cite{sim} based on a Seattle residence. We then use EnergyPlus to generate ground truth consumption data about the operation of the building. We add a linear term that represents the effect of demand response to create consumption data under a DR signal. The covariates of the building includes temperature, details of the fabrics of the roof, HVAC information, etc... These sample points are collected from many different time slots during various days.

Here we run both SLR and MLR to compare their performances. The magnitude of the demand response provided by the building is approximately $20\%$ of its total consumption. When the building receives a DR signal fairly often, for example about 50\% of the time periods, MLR outperforms SLR. But in the more likely scenario when the building rarely receives DR signal, SLR tends to out perform MLR. For example, when the building only receives DR signals 15\% of the time, the estimation error (of the average impact of the DR signal) is around 7.6\% from SLR, but 24.5\% from MLR.
This also motivates us to introduce the MCM estimator in Section~\ref{sec:lr}, which under many settings is more robust than MLR but more data efficient than SLR.

\begin{figure}[!t]
\centering
\includegraphics[width = 0.8\columnwidth]{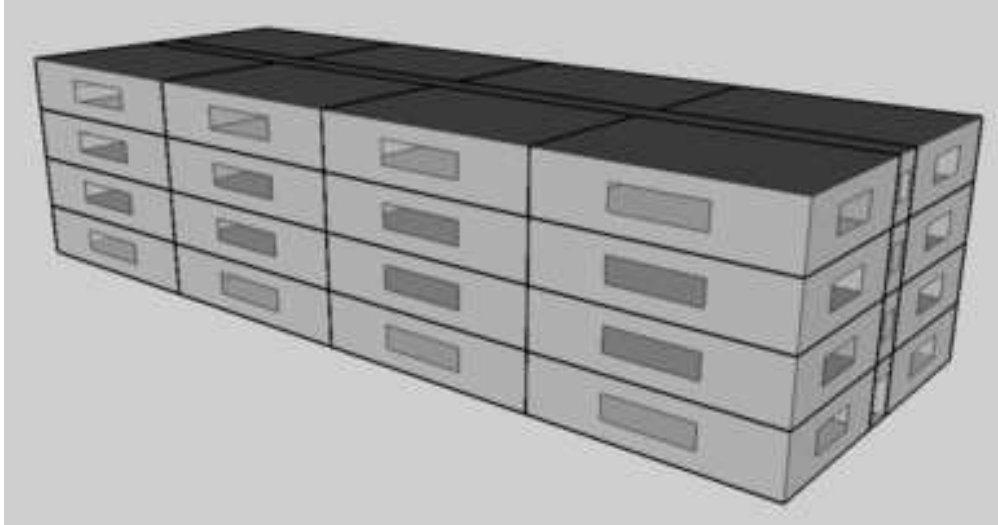}
\caption{An illustration of the building used in this paper. There are four floors in this multi-family residential building, with total floor area 3135 $m^2$.}
\label{building}
\end{figure}

\section{Problem Setup}\label{sec:setup}


\subsection{Model}\label{subsec: model}

Formally, we define the treatment effect of DR signals using the standard Neyman-Rudin model \cite{causalreview2}. In this model, the outcome $Y_i$ (in our case is the energy consumption) takes on one of the two values: $Y_i(0)$ or $Y_i(1)$, depending on whether a DR signal is sent or not.  {Note that an observation can represent a user when the samples are extracted from a pool of different users , or it can represent a time slot when the samples are from the same user but during different time periods.} Throughout this paper, we refer to sample $i$ as the particular observation with index $i$, where the context of the observation should always be clear from the problem setting. In total, there are $N$ observations.

Let $T_i$ be the binary indicator variable of the DR signal. Then the two values of $Y_i$, $Y_i(0)$ and $Y_i(1)$, are the potential outcomes either under the DR signal (when $T_i = 1$) or with no DR signal (when $T_i = 0$). Using the experimental design terms, we sometime refer to all of the observations that have $T=1$ as the treatment group and the observations with $T=0$ as the control group.

In most settings, the consumption $Y_i$ also depend on the covariates $\mathbf{x}_i$, and we can explicitly write $Y_i$ as a function of $T_i$ and $\mathbf{x}_i$ as:
\begin{subequations}
\begin{align}
Y_i & = T_iY_i(1) + (1-T_i)Y_i(0) \label{model1} \\
&= f(\mathbf{x}_i) + g(\mathbf{x}_i)T_i \label{model2},
\end{align}
\end{subequations}
where $f(\mathbf{x}_i) = Y_i(0)$, $g(\mathbf{x}_i)= Y_i(1) - Y_i(0)$. For the given $\mathbf{x}_i$, $g(\mathbf{x}_i)$ is the \emph{treatment effect} of the DR signal: it is the difference in consumption for an observation with and without the signal. Most of the time, we are after the \emph{averaged treatment effect (ATE)}, which is the empirical mean of $g(\mathbf{x}_i)$:
\begin{equation}
	{\bar{g}} = \frac{1}{N}\sum_{i=1}^{N}g_i(\mathbf{x}_i).
\end{equation}
In addition, we call $f_i(\mathbf{x}_i)$ the main effect for $i$.
In the following sections, we write $f(\mathbf{x}_i)$ as $f_i$ and $g(\mathbf{x}_i) $ as $ g_i$ for notational simplicity.

%
%


\subsection{Assumptions}

We assume a randomized trial scenario, where the treatment assignment strategy is independent of the covariates of the samples. This means that the treatment $T_i$ is a Bernoulli random variable which takes value 1 with probability $p$ and 0 with probability $1-p$, independent to everything else~\cite{causalreview2}.

 Note that this assumption might seem strict, but it is not beyond unreasonable. First, if a utility or operator do not know the treatment effect of DR, it is natural to model the process as a random trial. For example, utilities typically ask buildings for demand response based on the condition of the entire system and there is no direct relationship between the requests and the state of the building, so the requests can be thought as randomized treatments. Second, if one observation group is significantly preferred for demand response over another (e.g. commercial vs. residential users), the groups can be studied separately and modeled as receiving independent signals. Third, from the real data obtained from Pecan Street~( text messages of high prices), the users are coming in and out of the DR programs continuously without any obvious pattern, motivating the independence assumption.

\section{Linear Regression}\label{sec:lr}

%
 In this section, we describe three different linear methods (shown in Table~\ref{table: inout}) to estimate the ATE~($\bar{g}$): SLR on treatment variable, MLR on both the treatment variable and the covariates, and a regression using MCM introduced in \cite{covariate}.
 \begin{table}[!ht]
  \renewcommand{\arraystretch}{1.3}
  \caption{ Input for the least square estimation.}
  \centering
  \begin{tabular}{|c|c|}
    \hline \bfseries Model & \bfseries Input \\
    \hline
    SLR & $T_i$ and intercept  \\
    \hline
    MLR & $T_i$, $\mathbf{x}_i$ and intercept  \\
    \hline
    MCM & modified covariate $(T_i-p)\mathbf{x}_i$ and intercept \\
    \hline
  \end{tabular}
	\begin{tablenotes}
		\small
		\item $T_i$: binary indicator of DR signal, $p$: probability of receiving signal, and $\bd x$: covariates containing side information
	\end{tablenotes}

	  \label{table: inout}
 \end{table}

In all of these methods, we use least square (LS) to compute the estimator in the linear regression. Algorithm \ref{algo0} depicts the estimation procedure for the following canonical linear regression model:
\begin{equation}\label{linreg0}
\begin{aligned}
	Y_i & = \bd w_i ^{\top} \bm{\beta} + \epsilon_i\\
	& = \begin{bmatrix}
	T_i \\
	\mathbf{x}_i\\
	1
	\end{bmatrix} ^{\top}
	\begin{bmatrix}
	\beta^{(1)}\\
	\bm{\beta}^{(\mathbf{x})}\\
	\beta^{(0)}
	\end{bmatrix}
	 + \epsilon_i,
\end{aligned}
\end{equation}
{where $Y_i$ is observed power consumption, $\epsilon_i$ is the noise, and $\bd w_i$ is the regressor variable. The regressors $\bd w_i$ include a binary variable $T_i$ that indicates the treatment assignment, a possible set of covariates $\mathbf{x}_i$ such as temperature and building device information, and an intercept. Note that since an intercept is included, LS estimator is the same if we use $T_i - p$ as the regressor instead of $T_i$, where $p= \frac{1}{N}\sum_i T_i $.} In addition, parameter $\bm \beta$ is the weight associated with the regressor variable $\bd w_i$ where $\beta^{(1)}$ turns out to be the ATE.
	\begin{algorithm}[h]
		\caption{Linear regression estimation}
		\label{algo0}
		\begin{algorithmic}[1]
			\State \textbf{Input}: $N$ observations $(Y_i,\bd w_i)$ for $i=1,\dots,N$. Stacking into output vector $\bd Y=[Y_1,\dots,Y_n]$,  and regressor (input) matrix $\bd W$ with rows $\bd w_1^{\top}, \bd w_2^{\top},\dots,\bd w_N^{\top}$ variable. The weights are denoted by $\bm{\beta} \in \mathcal{R}^d$.
			\State The least square estimation for $\bm{\beta}$ is:
			\begin{equation}
				\hat{\bm \beta} = (\bm{W} ^{\top}\bm{\bm{W}} )^{-1}\bm{W} ^{\top}\bm{Y}.
			\end{equation}
			\State \textbf{Ouput}: The estimator $\hat{\beta}^{(1)}$, the first element of $\hat{\bm{\beta}}$, which will be the estimate of the ATE for appropriate definition of $\bm{w}$.
		\end{algorithmic}
	\end{algorithm}

For the three regressions in Table~\ref{table: inout}, their differences are in how the regressor ($\bd w$) and noise ($\epsilon$) are defined. The rest of this section show how they can be written in the canonical form in \eqref{linreg0} and how to find the ATE.


For all three regressions the LS estimator in Algorithm \ref{algo0} turns out to be consistent, i.e., as the number of samples grows, it will eventually converge to the correct value \cite{booklg}. Therefore, to compare their performances, we look at the \emph{variances} of the estimators, in particular in how fast the variances decrease as the number of samples increases.

\subsection{SLR on Treatment}
Definte $p = \frac{1}{N}\sum_i T_i$, which means that $pN$ samples are getting treatments (the DR signal). It can also be interpreted as the probability that each sample gets treated. The two interpretation are equivalent during the estimation procedure~\cite{Freedman2009}. We then rewrite (\ref{model2}) into the following form by centering the variables:
\begin{equation}\label{eq:linreg1}
\begin{aligned}
Y_i & = (T_i - p)\bar{g} + \bar{g} + \bar{f} + T_i(g_i-\bar{g}) + (f_i-\bar{f}) \\
& = \begin{bmatrix}
T_i-p\\
1
\end{bmatrix}^{\top} \begin{bmatrix} \bar{g} \\ \bar{g} + \bar{f} \end{bmatrix}  + \epsilon_i\\
& = \mathbf{w}_i^{\top} \bm{\beta}  + \epsilon_i
\end{aligned}
\end{equation}
where the noise $\epsilon_i = T_i(g_i-\bar{g}) + (f_i-\bar{f})$, with $\bar{g} = \frac{\sum_i g_i}{N}$ and $\bar{f} = \frac{\sum_i f_i}{N}$. In canonical form, $\bd w_i= \begin{bmatrix}
T_i-p & 1
\end{bmatrix}^T$ and $\bm \beta=\begin{bmatrix} \bar{g} & \bar{g} + \bar{f} \end{bmatrix}^T$. Let $Z_i=T_i-p$ and using the fact that $\sum_i Z_i = 0$, the LS estimator for the average treatment effect $\bar{g}$ is:
\begin{equation}\label{est1}
\begin{aligned}
\hat{\bar{g}}_{SLR}
& = \hat{\beta}^{(1)}\\
&  = \frac{\sum_iT_i(g_i+f_i)}{\sum_iT_i} - \frac{\sum_i(1-T_i)f_i}{\sum_i(1-T_i)}.
\end{aligned}
\end{equation}
The result in (\ref{est1}) shows that the estimator from SLR is the same as the difference-in-mean estimator~\cite{causalreview1}. It simply takes the difference between the average outcome between the treatment group and the control group to estimate the impact of the DR signal. However, this estimator is data inefficient, since it ignores any side information about the covariates that may be available.



\subsection{MLR on Treatment and Covariates}
Now suppose that we know some covariates of each customer $i$ and they are denoted by $\mathbf{x}_i$. A multiple linear regression model is carried out when both the treatment variable $T_i$ and $\mathbf{x}_i$ are included as regressors in \eqref{linreg0}:
\begin{equation}
	Y_i = \begin{bmatrix}
	T_i-p\\
	\mathbf{x}_i\\
	1 \\
	\end{bmatrix}^{\top} \bm{\beta} + \epsilon_i,
\end{equation}
and the estimate of ATE is again the first parameter of the estimate of $\bm{\beta}$: $\hat{\bar{g}}_{MLR}=\hat{\beta}^{(1)}$. However, the noise term maybe significant if the underlying true model is not linear.

The simulation results of the MLR estimator is presented in section \ref{sec:sim}. If we compare SLR and MLR by the reduction in the variance of the estimators, the latter does not always improve the estimation performance compared with the former. This phenomena is mainly due to the fact that if the underlying true relationship is nonlinear, modeling covariates has having a linear relationship with consumption introduces large noises, especially if $p$ is small. More detailed theoretical discussions are presented in Section \ref{sec:varanalysis}.

\subsection{Modified Covariate Method} \label{CM}
Even if including covariates directly into the linear model does not necessarily improve performance, it is still desirable to somehow use the covariate information. One possible improvement is to \emph{only assume linearity in the treatment effect for each customer $i$}. That is, only $g_i$ is linear in the covariates, and no assumption is made about how the main effect $f_i$ depend on the covariates.

We thus use a new method called \emph{Modified Covariate Method} (MCM), proposed in \cite{covariate}. This method assumes that the treatment effect is linear in the covariate, i.e., $g_i = \mathbf{x}_i^{\top}\gamma$, but we do not impose any conditions on $f_i$. In this case, the average treatment effect is:
\begin{equation}
\bar{g}_{MCM} = \frac{1}{N}\sum_i \mathbf{x}_i^{\top}\mathbb{\gamma}.
\end{equation}

We then have the following linear regression model: $Y_i = f_i + T_i\mathbf{x}_i^{\top}\gamma$.
Again, rewrite it in a canonical form:
\begin{equation}\label{MCM1}
\begin{aligned}
Y_i & = (T_i - p)\mathbf{x}_i^{\top}\bm \gamma + \bar{f} + p\bar{\mathbf{x}}^{\top}\bm \gamma + (f_i-\bar{f}) + p(\mathbf{x}_i-\bar{\mathbf{x}})^{\top}\bm{\gamma}\\
& = \begin{bmatrix}
(T_i-p) \mathbf{x}_i\\
1 \\
\end{bmatrix}^{\top} \begin{bmatrix} \bm{\gamma} \\ \beta^{(0)} \end{bmatrix} + \epsilon_i,
\end{aligned}
\end{equation}
where the noise $\epsilon_i = (f_i-\bar{f}) + p(\mathbf{x}_i-\bar{\mathbf{x}})^{\top}\bm \gamma$. We refer to $(T_i-p) \mathbf{x}_i$ as the \emph{modified covariate}. The LS estimator is still consistent~\cite{covariate} and we can estimate the treatment effect of the DR signal as $\hat{\bar{g}}_{MCM} = \frac{1}{N}\sum_i \mathbf{x}_i^{\top}\hat{\bm \gamma}$.

There are two reasons that MCM is potentially useful in estimating the ATE. First, MCM trades off between how much covariate information to use and model simplicity. The covariate information is only used with respect to the treatment effect, which is a weaker assumption than MLR. On the other hand, MCM captures more information by still using the covariates, which makes it more data efficient than SLR. Secondly, its formulation fits an interesting regime in the context of demand response. For example, the DR capability of a building could very well proportional to the covariates (e.g., proportional to the  occupancy of the building).

\section{Variance Analysis}\label{sec:varanalysis}
All the discussed estimators in Section \ref{sec:lr} yield a consistent estimate for ATE. However, their performances vary with respect to the respective variances, which we study in detail in this section.
Note that in the following analysis we focus on the case where covariate $\mathbf{x}_i$ is one dimensional, i.e., it is a scalar. The analysis of multi-dimensional covariates is tedious and does not provide additional intuitions. We demonstrate the results for multi-dimensional covariates using simulations.
To facilitate the analysis, we ignore the higher order terms in the calculation and approximate variances by their second order Taylor expansions. For details, please see the Appendix.



We first investigate the performances of SLR and MLR in cases when ${x}_i$ interacts nonlinearly with the consumption data. The main result is presented in Theorem \ref{theorem1}, where we adopt the notation that $\Cov(\cdot,\cdot)$ stands for the empirical covariance between two vectors, i.e., $\Cov (\mathbf{f}, \mathbf{x}) = \frac{\sum_i (f_i-\bar{f})(x_i-\bar{x})}{N}$, and $\Cov (\mathbf{g}, \mathbf{x}) = \frac{\sum_i (g_i-\bar{g})(x_i-\bar{x})}{N}$.

\begin{theorem}\label{theorem1}
	If $\Cov (\mathbf{g}, \mathbf{x}) = 0$ or $p = 0.5$, then MLR always yields a better performance than SLR. Otherwise, the performance of the two estimators depends on the value of $p$ and $k = \frac{\Cov (\mathbf{f}, \mathbf{x})}{\Cov (\mathbf{g}, \mathbf{x})}$. Assuming WLOG that the covariate $x_i$ has unit variance, then:
	\begin{equation}\label{MLRvar0}
	\begin{aligned}
	\Var(\hat{\bar{g}}_{SLR}-\bar{g})- \Var(\hat{\bar{g}}_{MLR}-\bar{g}) = \frac{\Delta}{p(1-p)N},
	\end{aligned}
	\end{equation}
	where:
	\begin{equation} \label{Delta0}
	\begin{aligned}
	\Delta  &= (\Cov(\mathbf{f}, \mathbf{x}))^2 + 2(1-p)\Cov (\mathbf{g}, \mathbf{x})\Cov (\mathbf{f}, \mathbf{x}) \\
	+ & (2p-3p^2)(\Cov (\mathbf{g}, \mathbf{x}))^2.
	\end{aligned}
	\end{equation}
\end{theorem}

%

The proof for Theorem \ref{theorem1} is left to Appendix. Note that $k$ represents the intensity that how $f$ and $g$ are correlated: a negative $k$ implies negative correlation whereas a positive $k$ implies positive correlation. The important conclusion from Theorem \ref{theorem1} is that $\Delta$ can be both negative and positive. The sign of $\Delta$ includes many factors such as the choice of $p$ and the correlation between the responses and the chosen covariate. The simplest case is when $g_i$ is a constant. In this case, $\Delta \geq 0$, which means that MLR is at least as good as SLR, given any arbitrary assignment probability $p$. When $p \neq 0.5$ and $g_i$ is not a constant across all samples, the sign of $\Delta$ depends both on $\mathbf{g}$ and $\mathbf{f}$. Figure \ref{fig0} depicts the $(p,k)$ region that results in a negative $\Delta$, which means that MLR is worse than SLR.
\begin{figure}[!t]
	\centering
	\includegraphics[width = 0.8\columnwidth]{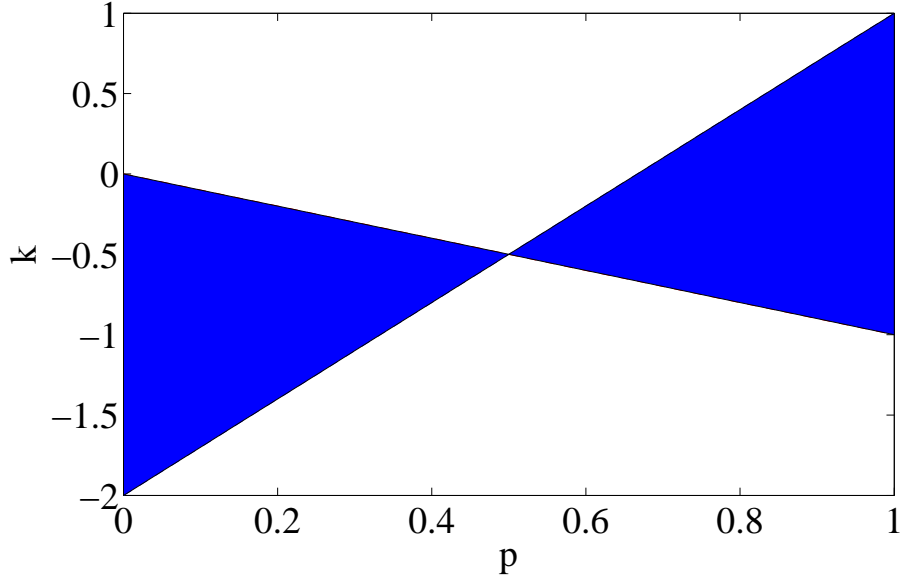}
	\caption{$(p,k)$ region (in blue) that results in a negative $\Delta$, where $p$ is the treatment assignment probability and $k = \frac{\Cov (\mathbf{f}, \mathbf{x})}{\Cov (\mathbf{g}, \mathbf{x})}$. The blue region also represent the pair of $(p,k)$ when SLR is better than MLR.}
	\label{fig0}
\end{figure}

From Fig.\ref{fig0}, we observe that the shaded region is asymmetric with respect to ratio $k$.
This is mainly due to the fact that the signal-to-noise ratio in the MLR model is not symmetric with respect to $k$.
In addition, from Fig.\ref{fig0}, we further observe that if $\mathbf{g}$ and $\mathbf{f}$ is positively correlated, i.e., $k >0$, then SLR performs better when a lot of samples get treatment ($p$ is large, and $\Delta <0$). This mens that the operator should trust the result by regressing on $T_i$, when the majority of samples are in the treatment group.

When $\mathbf{g}$ and $\mathbf{f}$ is negatively correlated, SLR performs better when a few samples get treatment ($p$ is small, and $\Delta <0$). This suggests that it is better to perform SLR when only a few samples are in the treatment group, i.e., when the treatment signals are scarce. For demand response, this is a regime of interest, since most customers receive relatively few DR signals, and the correlation between the main effect and treatment effect is negative. For example, considering temperature as the covariate, then the consumption increases together with temperature because of higher cooling needs. However, when temperature is high, people may be more reluctant to reduce consumption (turn off AC's in this case) to respond to DR signal.

From Fig. \ref{fig0}, we also observe that $\Delta$ changes sign when altering $p$ and fixing $k =0$. This is another interesting regime of demand response, when the treatment effect of each unit, i.e., $g_i$, is linear in the covariates, and the normal consumption is a constant ($f_i = f_c$ is a constant). In this case, the linear regression model is expressed in \eqref{eq:MCM_var}:
	\begin{equation}\label{eq:MCM_var}
	Y_i = T_i{x_i}\gamma + f_c + \epsilon_i,
	\end{equation}
where $g_i = x_i\gamma.$

This regime of interest in demand response. For example, the normal consumption without demand response does not appreciably but the treatment effect may depend on the covariates. In this case, $\Delta =  (2p-3p^2)(\Cov (\mathbf{g}, \mathbf{x}))^2$ meaning that as long as $0<p<\frac{2}{3}$, MLR is better than SLR.

\subsection{Comparison between SLR/MLR and MCM}\label{subsec2}
We can also analyze the performance of MCM with SLR and MLR in this specific case. We again restrict ourselves to one dimensional analysis (the covariate $x_i$ is one dimensional), as in last section. The main claim is shown in Remark \ref{remarkMCM}, with the assumption that the covariate follows a Gaussian distribution with unit variance and mean denoted by $\mu$. An illustration of a more general case with multi-dimensional covariate is shown at the end of this section.


\begin{remark}\label{remarkMCM}
	Consider the setup in~\eqref{eq:MCM_var}, where the normal consumption $f_i$ is a constant and the treatment effect to the signal $g_i$ is linear in the covariate $x_i$. Suppose that the covariate $x_i$ follows a Gaussian distribution with mean $\mu$ and unit variance. Then the performances of the three estimators depend both on the value of $p$ and $\mu$. We find that MCM performs the best when $p$ is relatively small and SLR performs the best when $p$ is relatively large. Otherwise, MLR yields the best estimator for ATE, i.e., $\bar{g}$, when $p$ takes on moderate values between 0 and 1. The variance of the proposed linear estimators are shown in \eqref{eq:MCMvarM}, \eqref{eq: SLRvarM} and \eqref{eq: MLRvarM}:
		\begin{subequations}
		\begin{align}
		\Var (\hat{\bar{g}}_{MCM} - \bar{g}) & = \Var ((\hat{\gamma}-\gamma)\frac{\sum_ix_i}{N} ) \nonumber \\
		& \approx \frac{\gamma^2\mu^2p^2(3+\mu^2)}{Np(1-p)(1+\mu^2)^2}, \label{eq:MCMvarM} \\
			\Var (\hat{\bar{g}}_{SLR}-\bar{g}) &= \frac{(1-p)^2\gamma^2}{p(1-p)N}, \label{eq: SLRvarM} \\
				\Var (\hat{\bar{g}}_{MLR}-\bar{g}) &= \frac{(2p-1)^2\gamma^2}{p(1-p)N}. \label{eq: MLRvarM}
		\end{align}
	\end{subequations}
		%
		%
		%

\end{remark}

To understand Remark \ref{remarkMCM} better, we consider an example with a nonzero $\mu$.
For more details, {please refer to appendix}.

%
Figure \ref{figMCM_SLR} compares MCM and SLR and Fig. \ref{figMCM_MLR} compares MCM and MLR. As can be seen from Fig. \ref{figMCM_SLR} and Fig. \ref{figMCM_MLR}, generally MCM performs better when $p$ is relatively small and worse when $p$ is relatively big. 
{To illustrate that the estimators are different, consider the case when $\mu = 1$. We find that when $p$ is close to 1, then SLR performs the best. On the contrary, when $p$ is close to 0, then MCM yields the smallest variance. When $p$ is around 0.5, MLR outperforms both SLR and MCM in terms of variance reduction. Therefore, if the model in \eqref{eq:MCM_var} correctly captures the consumption behavior, when the treatment signal is scarce, it is preferable to use MCM to estimate ATE. A more general case study is left to the appendix for interested readers.}

\begin{figure}[!t]
	\centering
	\includegraphics[width = 1.0\columnwidth]{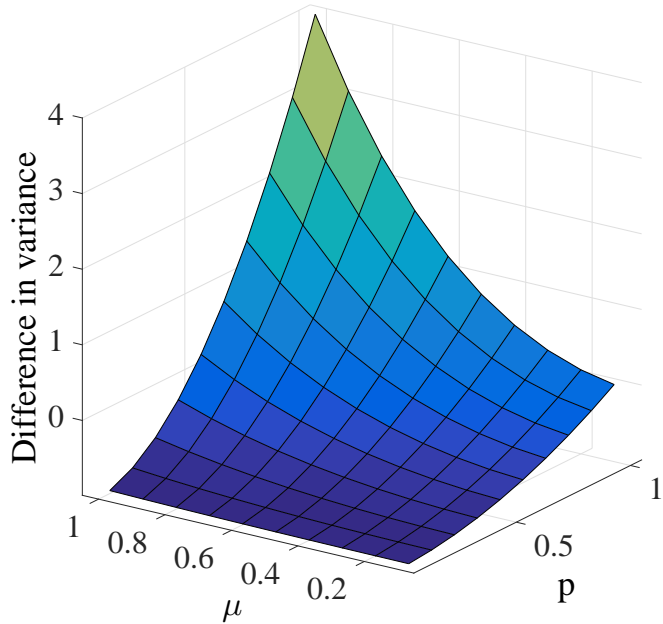}
	\caption{Difference of the variance obtained by MCM and SLR as a function of $\mu$ and $p$.}
	\label{figMCM_SLR}
\end{figure}

\begin{figure}[!t]
	\centering
	\includegraphics[width = 1.0\columnwidth]{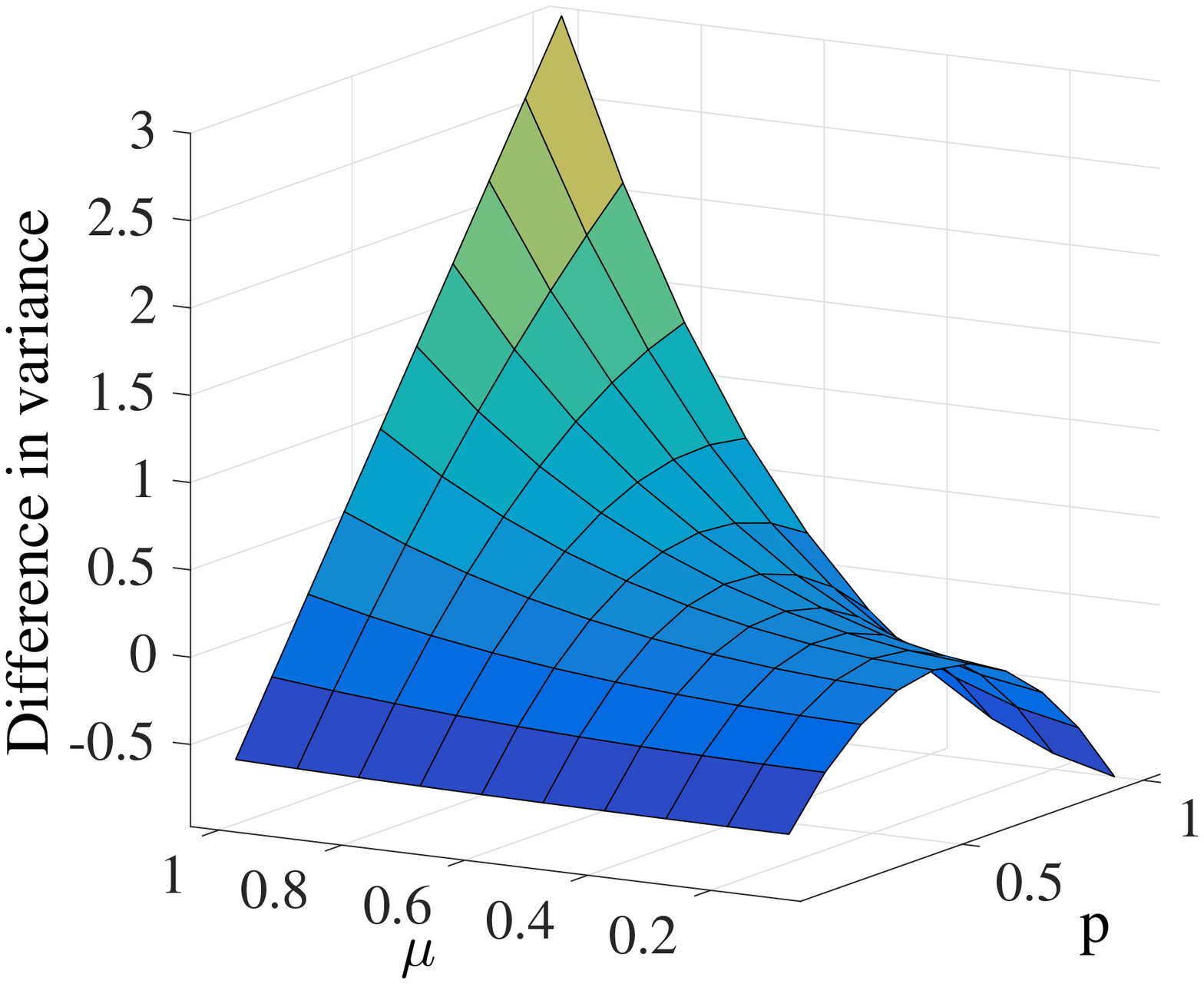}
	\caption{Difference of the variance obtained by MCM and MLR as a function of $\mu$ and $p$.}
	\label{figMCM_MLR}
\end{figure}

Overall, based on the discussion in Section \ref{sec:varanalysis}, the observation of the variance is summarized as the following: 
\begin{itemize}
	\item If $\mathbf{g}$ and $\mathbf{f}$ is positively correlated, then MLR performs better when $p$ is small.
	\item If $\mathbf{g}$ and $\mathbf{f}$ is negatively correlated, then MLR performs better when $p$ is large.
	\item If $\mathbf{g}$ is not correlated with $\mathbf{x}$, then MLR performs better.
	\item If $\mathbf{f}$ is not correlated with $\mathbf{x}$, then MLR performs better if $p$ is relatively small ($p<\frac{2}{3}$), otherwise SLR performs better.
	\item If the covariate $x_i$ is one dimensional, $g_i $ is linear in $x_i$, and $f_i $ is a constant across all sample $i$, MCM works better when a few samples get treatment signal, SLR works better when many samples get treatment signal, and MLR works better when a moderate number of samples get treatment signal.\footnote{This requires a technical condition of $\sum_i x_i \neq 0$, which is easily satisfied since typical covariates used in demand response such as temperature or occupancy are not zero centered.} 
\end{itemize}

\section{Case Study}\label{sec:sim}

In this section, we conduct experiments on data from three sources: synthetic data, building data, and data from Pecan Street \cite{pecan}.
\subsection{Synthetic Data}
We first generate data from two linear models, with one dimensional covariate $x_i$ drawn from a Gaussian distribution with unit mean and unit variance. The models are:
\begin{subequations}
	\begin{align}
		\label{syntheticlinear1}
			Y_i &= {{x}_i\alpha_1} + {x}_i\alpha_2 T_i ,\\
		\label{syntheticlinear2}
		Y_i &= \alpha_0 + {x}_i\alpha_2 T_i .
	\end{align}
\end{subequations}

We take $\alpha_1 = 6$ and $\alpha_2 = 20$, so that $ k = 0.3$ for (\ref{syntheticlinear1}). From Theorem \ref{theorem1}, this means that SLR performs better if $p>0.77$. For (\ref{syntheticlinear2}) $k = 0$, since the main effect is a constant. We further set $\alpha_0 = 20$, and the simulated covariate to be centered at 1. {This is the regime discussed at in Section \ref{sec:varanalysis}, with the case where main effect $f_i$ is a constant and that treatment effect is linear, i.e., $g_i  = {x}_i\alpha_2 $. }The results are shown in Fig. \ref{figlinear}.

\begin{figure}[!t]
	\centering
	\includegraphics[width = 1.0\columnwidth]{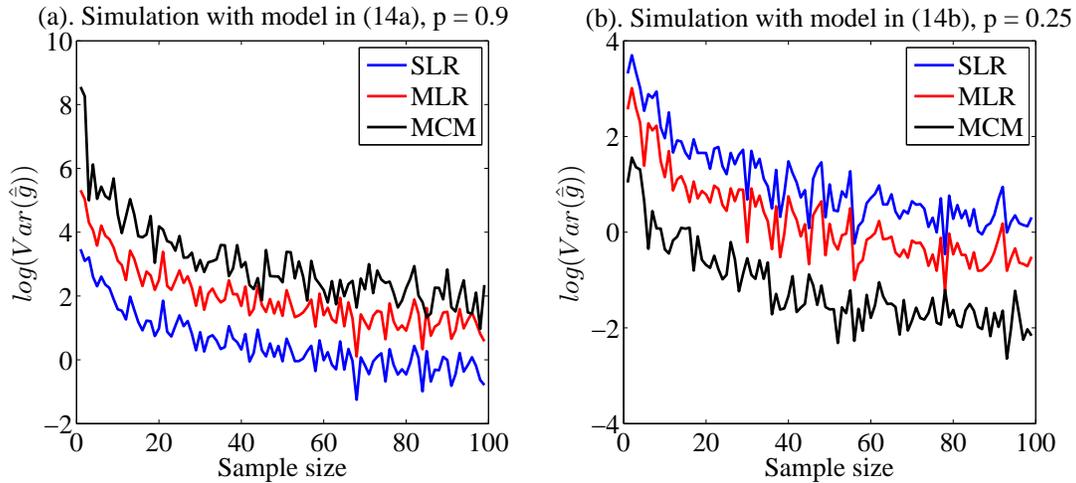}
	\caption{Variance of the three estimators of linear model in semi-log scale. Left figure for (\ref{syntheticlinear1}) with $p = 0.9$, right figure for (\ref{syntheticlinear2}) with $p = 0.25$.}
	\label{figlinear}
\end{figure}

As can be seen from Fig. \ref{figlinear}, the performances of the three estimators depend on both the assignment probability $p$ and the structure of the underlying model. The results from Fig. \ref{figlinear}(a) supports our claim that when the treatment assignment probability is high, i.e., $p = 0.9$, SLR performs the best with model in (\ref{syntheticlinear1}). As for the model presented in (\ref{syntheticlinear2}), {according to the discussions in Section \ref{sec:varanalysis}}, MCM is the best estimator with least variance, which is demonstrated in Fig. \ref{figlinear}(b).

\subsection{Building data}
	In this section, we validate the claim in the paper via building data generated using EnergyPlus~\cite{sim}. Using this software, we generate building covariates such as environment temperature, number of occupants, appliances scheduling, etc. {They are denoted by a vector $\mathbf{x}_i$ for each single observation $i$. The building model is illustrated in Fig. \ref{building} at Section \ref{sec:motiv}. In this specific residential building, there are four floors with eight apartment located at each floor and office located at the first floor. The total floor area is 3135 squared meters.}

In this simulation, the sample index $i$ denotes different time slots of the building consumption model. The treatment $T_i$ is added to $f_i$, so the final output of the simulated building consumption under DR signal is $Y_i = f_i + g_iT_i$, where $g_i$ is the treatment effect of building $i$. To validate the claims in this paper, we consider two scenarios: 1) $g_i$ is a constant (when treatment effect is a constant across all sample points); 2)
$g_i$ is linear in $\mathbf{x}_i$ (when treatment effect is linear in the covariates for each sample point).

	\begin{table}[!ht]
		\renewcommand{\arraystretch}{1.3}
		\caption{ $\Var \hat{\bar{g}}$ (normalized) based on EnergyPlus data, where $p = 0.5$ and $g_i$ is a constant.}
		\centering
		\begin{tabular}{|c|c|c|}
			\hline \bfseries SLR  & \bfseries MLR & \bfseries MCM\\

			  \hline
			  1.000  &  0.100  &  1.506 \\
			\hline
		\end{tabular}
		 \label{table3}
	\end{table}

	The results of the two scenarios are shown in Table \ref{table3} ($g_i$ is a constant) and Table \ref{table4} ($g_i$ is linear in $\mathbf{x}_i$).
	When $g_i$ is a constant, as shown in Table \ref{table3}, MLR performs the best among all the three methods because the variance of ATE is the lowest. This validates the proposed claim in Theorem \ref{theorem1} that as long as $g_i$ is a constant, MLR is always the best linear estimator, for all possible values of $p$.

	However, it is not always true that the treatment effect $g_i$ is a constant for all types of buildings, we therefore generate another set of DR simulations with $g_i = \mathbf{x}_i^{\top}\gamma$, i.e., that the treatment effect is linear in the covariates of the particular building $i$. 
	We compare the estimators based on their variance and the results are illustrated in Table \ref{table4}, where the treatment assignment probability $p = 0.15$. From Table \ref{table4}, we observe the opposite as to Table \ref{table3}, that MLR yields the worst performance among all three estimators. In addition, MCM and SLR have similar performances. {Therefore, contrary to the case in Table \ref{table3}, the utility company should not blindly trust the results from a seemingly more powerful model, i.e., MLR, and should be careful about the interactions between covariates when modeling consumption behavior.} In all, the observations from Table \ref{table3} and Table \ref{table4} validate the claim in this paper, that it is not always good to conduct a full linear regression with all the covariates in the model, since the inclusion of those covariates may lead to a larger noise in the model if the linear regression model is not correctly defined.

	\begin{table}[!ht]
		\renewcommand{\arraystretch}{1.3}
		\caption{$\Var \hat{\bar{g}}$ (normalized) based on EnergyPlus data, where $p = 0.15$ and $g_i$ is linear in $\mathbf{x}_i$.}
		\centering
		\begin{tabular}{|c|c|c|}
			\hline
		  \bfseries SLR  & \bfseries MLR & \bfseries MCM \\
		\hline
		1.000 & 3.191 & 1.020 \\
		\hline
		\end{tabular}
		 \label{table4}
	\end{table}

\subsection{Pecan Street Data}
In this section, we test the estimators on data from Pecan Street~\cite{pecan}. In the tests, we treat the high price signals as treatments and the user index as the sample point index. The outcomes are these users' consumption data. To compose the treatment group and control group, we extract the high price signal and include users whose consumption data is available at that time into the treatment group, and include the other users into the control group. For the users in the control group, we find their consumption data at the same hour in the date closest to the high price signal date. This mimics the situation where the signals are randomized assigned, since each user has some chance of receiving a specific signal. Temperature is a primary regressor that researchers use in practice, so we include temperature into the linear regression model as covariate. Other covariates such as appliance information can easily be added.

Since we do not know the true ATE and the true model for observational data, we use $p$-values associated with the $t$-test and the $F$-test to make comparisons with Pecan Street data. These tests are hypothesis tests for linear regression models with Gaussian noise~\cite{booklg}. The difference between $t$-test and $F$-test is that $t$-test only examines whether including one particular regressor significantly improve the model whereas the $F$-test examines whether including all regressors significantly improve the model. Suppose that we just include one covariate into the model, the MLR model for sample $i$ is in the following form:
\begin{equation}
	\begin{aligned}
	Y_i & = \bd w_i ^{\top} \bm{\beta} + \epsilon_i\\
	& = \begin{bmatrix}
	T_i -p\\
	{x}_i\\
	1
	\end{bmatrix} ^{\top}
	\begin{bmatrix}
	\beta^{(1)}\\
	{\beta}^{({x})}\\
	\beta^{(0)}
	\end{bmatrix}
	+ \epsilon_i,
	\end{aligned}
\end{equation}
where $\bar{g} = \beta^{(1)}$ is the ATE to a specific DR signal, $T_i-p$ is the centered binary indicator variable for DR signal and ${x}_i$ is the covariate. The null hypothesis for the $t$-test in this regression model is given as:

\begin{equation}
H_0:	 \beta^{(1)} = 0,
\end{equation}
and for the $F$-test:
\begin{equation}
H_0:	 \beta^{(1)} = {\beta}^{({x})} = 0.
\end{equation}

For both tests, we examine the significance by setting a confidence level $\alpha$, normally taken as 0.05, or more strictly as 0.01. While comparing the values of a certain statistic under different models does not seem intuitive, we can alternatively resort to $p$-value, which is defined as the probability of obtaining the observed(or more extreme) result under the null hypothesis. Higher $p$-values suggest that the null hypothesis is true, whereas smaller $p$-values suggest the opposite. We then can compare the $p$-value to interpret the significance test under different regression models.

In addition, the treatment group has 100 observations and the control group has 500 observations.
{The estimation of ATE from the three models are shown in Table \ref{table_est_pecan}. As can be seen from Table \ref{table_est_pecan}, the estimation obtained by MLR is lower than that from SLR, we then perform MCM and its estimation of ATE is again higher than MLR. These results suggest that MLR might have returned an underestimate of ATE, even with extra regressors in the regression model.}
\begin{table}[!ht]
	\renewcommand{\arraystretch}{1.3}
	\caption{Estimation results for pecan street data.}
	\centering
	\begin{tabular}{|c|c|}
		\hline
		\bfseries  & \bfseries Estimation of ATE \\
		\hline
		$\hat{\bar{g}}_{SLR}$  & 1.16 \\
		\hline
		$\hat{\bar{g}}_{MLR}$ & 0.59 \\
		\hline
		$\hat{\bar{g}}_{MCM}$ & 0.90 \\
		\hline
	\end{tabular}
	\label{table_est_pecan}
\end{table}

{We further examine the performance of the models by significance tests.} The results are shown in Table \ref{table1}. From the results in Table \ref{table1}, we can see that the $p$-value with the $F$-test for all methods is generally small, meaning that the consumption data cannot be explained by just an intercept. {However, the $p$-value associated with the $t$-test is the highest for MLR, suggesting the insignificance of regressing on the treatment variable if the threshold is 0.01.} This is mainly due to the lack of information on how the covariates interact with consumption data and that the treatment group is much bigger than the control group. So if an utility uses MLR to estimate the ATE, it may conclude that the DR program is ineffectual by mistake. It is then beneficial to run more significance tests in SLR and MCM to gain more insights to the significance of the treatment effect. Also, we argue that although including covariates into the model may seem to improve prediction (smaller $p$-value for the $F$-test), it does not necessarily lead to a better inference.
\begin{table}[!ht]
	\renewcommand{\arraystretch}{1.3}
	\caption{Significance results for pecan street data.}
	\centering
	\begin{tabular}{|c|c|c|}
		\hline
		\bfseries  & \bfseries $p$-value for $t$ test  & \bfseries $p$-value for $F$ test \\
		\hline
		SLR  & 2.7e-7 & 2.7e-7\\
		\hline
		MLR  & 1.4e-2 & 1.4e-4\\
		\hline
		MCM & 2.9e-09 & 2.9e-09\\
		\hline
	\end{tabular}
	\label{table1}
\end{table}

Nevertheless, if the treatment effect is linear in the covariates, we can always proceed with MCM.  From Table \ref{table1} we can see that the $p$-value for $t$-test with MCM is small, suggesting that we should regress on the modified covariate. In this case, the SLR and MLR estimates agree, providing confidence that the ATE is close to 1 rather than the value of 0.59 as suggested by MLR. 

\section{Conclusion and Future Work}\label{sec:conclusion}
In this paper, we estimate the average treatment effect~(ATE) of demand response programs, defined as the average change in consumption when users receive DR signals.
We derive linear estimators for ATE through simple linear regression, multiple linear regression and modified covariate method. The simulation results show that although including more covariates may be good for prediction purposes (as in multiple linear regression), the performance of the estimators depend both on the assignment probability and the correlation between the effect and the covariate.
 Thus, the interactions between the covariates and the demand response signal must be carefully modeled and we provide practical guidance on which estimators should be used based on both synthetic and real data. This work provides a framework for further research in applying causal inference in analyzing consumption data and DR interventions.

\bibliographystyle{IEEEtran}
\bibliography{refs}
\section*{Appendix}
\subsection{Consistency of SLR estimator}

The estimator in (\ref{est1}) is consistent. To show this, we divide both the nominator and denominator by $N$. Since $T_i$'s are i.i.d. random variables, the denominator converges to a constant value by the strong law of large numbers:
\begin{equation}\label{convergence1}
\sum_i (T_i-p)^2/N
\overset{a.s.}{\rightarrow} \ p(1-p).
\end{equation}

Similarly, the nominator converges to zero due to the strong law of large numbers. Based on these convergence analysis, and with Slutsky's theorem~\cite{theorem}, we conclude that the estimator is consistent.

\subsection{Second order approximation}

The following lemma states the second order approximation~\cite{bookvar}:
\begin{lemma}\label{Lemma1}
	The variance of a ratio of two random variables can be approximated by:
	\begin{equation}
	\Var \frac{X}{Y} \approx \frac{\Var X}{(\E Y)^2} +2\frac{-\E X}{\E ^3Y}\Cov(X,Y) + \frac{\E^2X}{\E^4Y}\Var Y.
	\end{equation}

	If the numerator has zero mean, i.e., $\E X = 0$, then the variance is simplified as:
	\begin{equation}
	\Var(\frac{X}{Y}) \approx \frac{\Var(X)}{(\E Y)^2}.
	\end{equation}
	\end{lemma}

	Note that in Lemma 1, both $X$ and $Y$ are random variables. We can also define the covariance (variance) with respect to two vectors. For instance, given two vectors $\mathbf{g}$ and $\mathbf{x}$, the normalized vectors are $\mathbf{g} - \bar{\mathbf{g}}$ and $\mathbf{x} - \bar{\mathbf{x}}$. The covariance between $\mathbf{g}$ and $\mathbf{x}$ is $(\frac{\mathbf{g} - \bar{\mathbf{g}})^{\top}(\mathbf{x} - \bar{\mathbf{x}})}{N}$, where $N$ is the length of the two vectors.

\subsection{Proof of Theorem \ref{theorem1}}

Before proving Theorem \ref{theorem1}, we give the close form formulation of the estimator of ATE under both SLR and MLR.

First we derive the variance of the estimator from SLR. Recall that the estimator can be written as:
\begin{equation}
\hat{\bar{g}}_{SLR} =  \bar{g} + \frac{\sum_i (T_i-p)T_i(g_i-\bar{g})+\sum_i (T_i-p)(f_i-\bar{f})}{\sum_i (T_i-p)^2}.
\end{equation}

The nominal variance for $\hat{\bar{g}}_{SLR}$ can be expressed as:
\begin{equation}
\begin{aligned}
& \Var(\hat{\bar{g}}_{SLR}-\bar{g}) \\
& = \Var \frac{\sum_i (T_i-p)T_i(g_i-\bar{g})+\sum_i (T_i-p)(f_i-\bar{f})}{\sum_i (T_i-p)^2}.
\end{aligned}
\end{equation}

Since $\E \sum_i (T_i-p)T_i(g_i-\bar{g})+\sum_i (T_i-p)(f_i-\bar{f}) = 0$, according to Lemma \ref{Lemma1}, the variance is then approximated as:
\begin{equation}\label{SLRvar}
\begin{aligned}
& \Var(\hat{\bar{g}}_{SLR}-\bar{g}) \\
& \overset{(a)}{\approx} \frac{\Var \sum_i (T_i-p)T_i(g_i-\bar{g})+\sum_i (T_i-p)(f_i-\bar{f}) }{(\E \sum_i (T_i-p)^2)^2} \\
& = \frac{\Var \sum_i (T_i-p)T_i(g_i-\bar{g})+\sum_i (T_i-p)(f_i-\bar{f}) }{p^2(1-p)^2N^2} \\
& \overset{(b)}{=} \frac{ (1-p)^2 \Var \mathbf{g} + \Var \mathbf{f}  + 2(1-p)\Cov (\mathbf{g},\mathbf{f})}{p(1-p)N} ,
\end{aligned}
\end{equation}
where $(a)$ follows from Lemma 1, and $(b)$ follows from the fact that $T_i$ is the only random variable and that $f_i$ and $g_i$ takes some fixed value. We also denote $\Var \mathbf{g} = \frac{\sum_i (g_i-\bar{g})^2}{N}$, $\Var \mathbf{f} = \frac{\sum_i (f_i-\bar{f})^2}{N}$, and $\Cov (\mathbf{g}, \mathbf{f} )= \frac{\sum_i (g_i-\bar{g})(f_i-\bar{f})}{N}$ .

With this expression, we can proceed to compute the variance of the estimator from MLR and prove Theorem \ref{theorem1}.

\begin{proof}[Proof of Theorem \ref{theorem1}]

To derive the expression of the variance from MLR, we first assume that the covariate is one dimensional, i.e., $\mathbf{x}_i = x_i$, and without loss of generality, we also assume that $x_i$ has unit variance. The variance for the estimator in MLR can be expressed as:
\begin{equation}\label{MLRvar}
\begin{aligned}
\Var(\hat{\bar{g}}_{MLR}-\bar{g}) = \Var(\hat{\bar{g}}_{SLR}-\bar{g}) - \frac{\Delta}{p(1-p)N},
\end{aligned}
\end{equation}
where:

\begin{equation} \label{Delta}
\begin{aligned}
\Delta  &= (\Cov(\mathbf{f}, \mathbf{x}))^2 + 2(1-p)\Cov (\mathbf{g}, \mathbf{x})\Cov (\mathbf{f}, \mathbf{x}) \\
+ & (2p-3p^2)(\Cov (\mathbf{g}, \mathbf{x}))^2.
\end{aligned}
\end{equation}


The expression in (\ref{MLRvar}) bridges the variance between SLR and MLR. We can compare the performances of these two estimators on the sign of $\Delta$. If $\Delta >0$, then MLR has a smaller variance than SLR and vise versa. This observation leads to Theorem \ref{theorem1}, where we establish the comparison of performances between regression only on treatment indicator and on more information. This comparison depends on the assignment probability $p$ and the correlation of the covariate between the response to the treatment and the response without treatment. The surprising result shown in Theorem \ref{theorem1} is that it is not always good to include more covariates into the regression to obtain a better estimate of ATE.

With the expression in (\ref{Delta}), we now proceed to prove the claims shown in Theorem \ref{theorem1}.

		Write $\Delta$ as a quadratic function of $p$:
		\begin{equation}
		\begin{aligned}
		\Delta & = -3(\Cov (\mathbf{f}, \mathbf{x}))^2p^2 \\
		+ & [2(\Cov (\mathbf{g}, \mathbf{x}))^2-2\Cov (\mathbf{f}, \mathbf{x})\Cov (\mathbf{g}, \mathbf{x})]p \\
		+ & [(\Cov (\mathbf{f}, \mathbf{x}))^2 + 2\Cov (\mathbf{f}, \mathbf{x})\Cov (\mathbf{g}, \mathbf{x})]. \\
		\end{aligned}
		\end{equation}

		If $\Cov (\mathbf{g}, \mathbf{x})= 0$, then $\Delta = (\Cov (\mathbf{f}, \mathbf{x}))^2 \geq 0$.

		If $p = 0.5$, $\Delta = (0.5 \Cov (\mathbf{g}, \mathbf{x}) + \Cov (\mathbf{f}, \mathbf{x}))^2 \geq 0$.

		Otherwise, we have $\Delta < 0$ if $p$ satisfies the following condition:
		\begin{equation}
		\text{max} (\frac{2+ k}{3}) < p <1,
		\end{equation}
		or
		\begin{equation}
		o<p<\text{min}(\frac{2+k}{3}, -k).
		\end{equation}

		The region of ($p$, $k$) where $\Delta \leq 0$ is shown in Fig. \ref{fig0}.

\end{proof}

\subsection{Validation for Remark \ref{remarkMCM}}
To validate the claim in Remark \ref{remarkMCM}, we first consider the linear regression model in the following form:

\begin{equation}\label{modelMCM_constant}
Y_i = T_i{x_i}\gamma + f_c + \epsilon_i.
\end{equation}

In this case, we fix $f(x_i) = f_c$ to be a constant and that the treatment effect $g_i$ is linear in covariate $x_i$. The covariate $x_i$ is assumed to have unit empirical variance and empirical mean as $\mu$. Based on (\ref{MCM1}), the estimator for $\gamma$ is :
\begin{equation}\label{gamma}
\begin{aligned}
\hat{\gamma} & = \gamma + \frac{\sum_i[N(T_i-p)x_i-\sum_j(T_j-p)x_j]p(x_i-\mu)\gamma}{N\sum_i(T_i-p)^2x_i^2 - (\sum_j(T_j-p)x_j)^2} .\\
\end{aligned}
\end{equation}

Again, notice that $\E \{ \sum_i(N(T_i-p)x_i-\sum_j(T_j-p)x_j)px_i\gamma \}= 0$, based on \eqref{gamma}, we have:
\begin{equation}\label{MCMvar}
\begin{aligned}
\Var (\hat{\gamma}-\gamma) & \overset{(a)}{\approx} \frac{\Var \sum_i[(T_i-p)x_i-\sum_j\frac{(T_j-p)x_j}{N}]p(x_i-\mu)\gamma}{\{\E (\sum_i(T_i-p)^2x_i^2 - (\sum_j\frac{(T_j-p)x_j}{N})^2N)\}^2} \\
& \overset{(b)}{=}  \frac{\Var \sum_i[(T_i-p)x_i-\sum_j\frac{(T_j-p)x_j}{N}]p(x_i-\mu)\gamma}{\{\sum_i \{\Var [(T_i-p)x_i] - \Var [\sum_j\frac{(T_j-p)x_j}{N}]\}\}^2},
\end{aligned}
\end{equation}
where $(a)$ follows from Lemma 1 and the division by $N$ from both the numerator and the denominator. Equality $(b)$ is due to the fact that $\E (T_i-p) = 0$, so $\E(T_i-p)^2 = \Var (T_i-p)$. Same applies to  $\E \sum_j\frac{(T_j-p)x_j}{N}$. So both the numerator and denominator in (\ref{MCMvar}) involves calculating the difference between the variance of $(T_i-p)x_i$ and $\sum_j\frac{(T_j-p)x_j}{N}$. We will simplify this calculation by ignoring the term $\sum_j\frac{(T_j-p)x_j}{N}$ in both the numerator and the denominator. This operation is due to the fact that the variance of $\sum_j\frac{(T_j-p)x_j}{N}$ decays as $O(\frac{1}{N})$. When compared with $\Var [(T_i-p)x_i]$, they constitute higher order terms and can be ignored.

This simplification reduces (\ref{MCMvar}) to a more much compact form:

\begin{equation}
\begin{aligned}
\Var (\hat{\gamma}-\gamma)
& \overset{(a)}{\approx} \frac{\Var \sum_i[(T_i-p)x_i]p(x_i-\mu)\gamma}{\{\sum_i \{\Var [(T_i-p)x_i]\}\}^2} \\
& = \frac{\gamma^2p^3(1-p)\sum_i x_i^2(x_i-\mu)^2}{N^2p^2(1-p)^2(\frac{\sum_ix_i^2}{N})^2},
\end{aligned}
\end{equation}
where $(a)$ again follows from Lemma 1.

As we assume that the covariate follows a Gaussian distribution and has unit variance, i.e., $\frac{\sum_i (x_i-\mu)^2}{N} = 1$, then the variance of the estimator for $\gamma$ can be written as:
\begin{equation} \label{vargamma}
\begin{aligned}
\Var (\hat{\gamma}-\gamma) & \approx \frac{\gamma^2p^2(3+\mu^2)}{Np(1-p)(1+\mu^2)^2}.
\end{aligned}
\end{equation}

Based on (\ref{vargamma}), we construct the variance for the estimation of $\bar{g}$ in MCM:
\begin{equation}\label{MCMvarM}
\begin{aligned}
\Var (\hat{\bar{g}}_{MCM} - \bar{g}) & = \Var ((\hat{\gamma}-\gamma)\frac{\sum_ix_i}{N} )\\
& = \Var (\hat{\gamma}-\gamma)\mu^2 \\
& \approx \frac{\gamma^2\mu^2p^2(3+\mu^2)}{Np(1-p)(1+\mu^2)^2},
\end{aligned}
\end{equation}

With (\ref{MCMvarM}), (\ref{SLRvar}), and (\ref{MLRvar}), we can compare the performance of MCM with that of SLR/MLR. Note that when $f_i$ is a constant and $g_i = x_i\gamma$, the variance of SLR/MLR estimators are reduced to a simpler form which leads us to the following forms:
\begin{equation}\label{SLRvarM}
\Var (\hat{\bar{g}}_{SLR}-\bar{g}) = \frac{(1-p)^2\gamma^2}{p(1-p)N}
\end{equation}
and
\begin{equation}\label{MLRvarM}
\Var (\hat{\bar{g}}_{MLR}-\bar{g}) = \frac{(2p-1)^2\gamma^2}{p(1-p)N}.
\end{equation}

\subsection{Performance comparison in a non-linear case}
{Apart from the considered cases, a more general and interesting case is when $f_i$ in is non-linear in $x_i$. For example, the consumption is non-linear in temperature, or the thickness of the wall fabric. As an illustrating example of the nonlinearity, we construct the following outcome:}
\begin{equation}\label{outcome2}
Y_i = \sqrt[4]{|\sum_{j=1}^{d}{x}_{i,j}^3\gamma_j|} + \mathbf{x}_i^{\text{T}}\alpha T_i.
\end{equation}

{It is hard to derive a close form formulation of the variance of the three estimators for general non-linear terms. Alternatively, we simulate the empirical variance of the three estimators in (\ref{outcome2}) by discretized values of $p$ in the range of (0,1) with an interval 0.1. We again fix the mean of $x_i$, i.e., $\mu$, to be 1.} 
\begin{table}[!ht]
	\renewcommand{\arraystretch}{1.3}
	\caption{Results for performance comparison of the three estimators with the outcome model shown in (\ref{outcome2}). }
	\centering
	\begin{tabular}{|c|c|c|c|}
		\hline
		\bfseries  & \bfseries Best  & \bfseries Medium & \bfseries Worst \\
		\hline
		$p=0.1$  & MCM & MLR & SLR \\
		\hline
		$p=0.2$  & MCM & MLR & SLR\\
		\hline
		$p=0.3$  & MLR & MCM & SLR \\
		\hline
		$p=0.4$  & MLR & MCM & SLR \\
		\hline
		$p=0.5$  & MLR & SLR & MCM \\
		\hline
		$p=0.6$ & MLR & SLR & MCM \\
		\hline
		$p=0.7$  & MLR & SLR & MCM \\
		\hline
		$p=0.8$  & SLR & MLR & MCM \\
		\hline
		$p=0.9$  & SLR & MLR & MCM \\
		\hline
		\end{tabular}
		\label{table01}
		\end{table}

		From Table \ref{table01}, we observe that MCM works the best when $p$ is close to 0 and SLR when $p$ is close to 1, otherwise MLR yields the best estimator. The only difference is that the thresholding values of $p$ changes when $f_i$ is no longer a constant. Thus, in a general case, we prefer to adopt MCM when the treatment group is much smaller than the control group, and SLR in the inverse case. When the treatment group and the control group is of similar size, we prefer using MLR to make an estimate.

\subsection{Simulation on more synthetic data}
We simulate data from more complex models and compare the estimators' performances under four models, i.e., a linear model, a non-linear model, a non-linear model with constant $f_i$, and a model with non-linear $f_i$ but linear $g_i$. The covariates $\mathbf{x}$'s in theses cases are multi-dimensional and they are i.i.d. samples from an arbitrary joint Gaussian distribution with off-diagonal elements as zeros in the covariance matrix. The probability of a treatment assignment is set to be 0.8, 0.9, 0.75 and 0.1 respectively for these four models (model 1 to model 4) presented in (\ref{synthetic1}) to (\ref{synthetic4}), where $d$ is the dimension of the covariate vector $\mathbf{x}_i$.
\begin{subequations}
	\begin{align}
		\label{synthetic1}
		Y_i &= {\mathbf{x}_i^{\text{T}}\gamma} + \mathbf{x}_i^{\text{T}}\theta T_i ,\\
		\label{synthetic2}
		Y_i &= \sqrt[4]{|\sum_{j=1}^{d}{x}_{i,j}^3\alpha_j|} + ({\sum_{j=1}^{d}{x}_{i,j}^2\theta_j + \sum_{j=1,k\neq j}^{d}{x}_{i,j}{x}_{i,k}\theta_{j,k}} )T_i ,\\
		\label{synthetic3}
		Y_i &= \theta_0 + ({\sum_{j=1}^{d}{x}_{i,j}^2\theta_j + \sum_{j=1}^{d}\sum_{k\neq j}^{d}{x}_{i,j}{x}_{i,k}\theta_{j,k}} )T_i ,\\
		\label{synthetic4}
		Y_i &= \sqrt[4]{|\sum_{j=1}^{d}{x}_{i,j}^3\theta_j|} + \mathbf{x}_i^{\text{T}}\theta T_i.
		\end{align}
		\end{subequations}

		The variance decay of the three estimators is shown in Fig.\ref{fig1}. 
		We compare the performance of the estimators based on the magnitude of this variance.
		\begin{figure}[!t]
			\centering
			\includegraphics[width = 1.0\columnwidth]{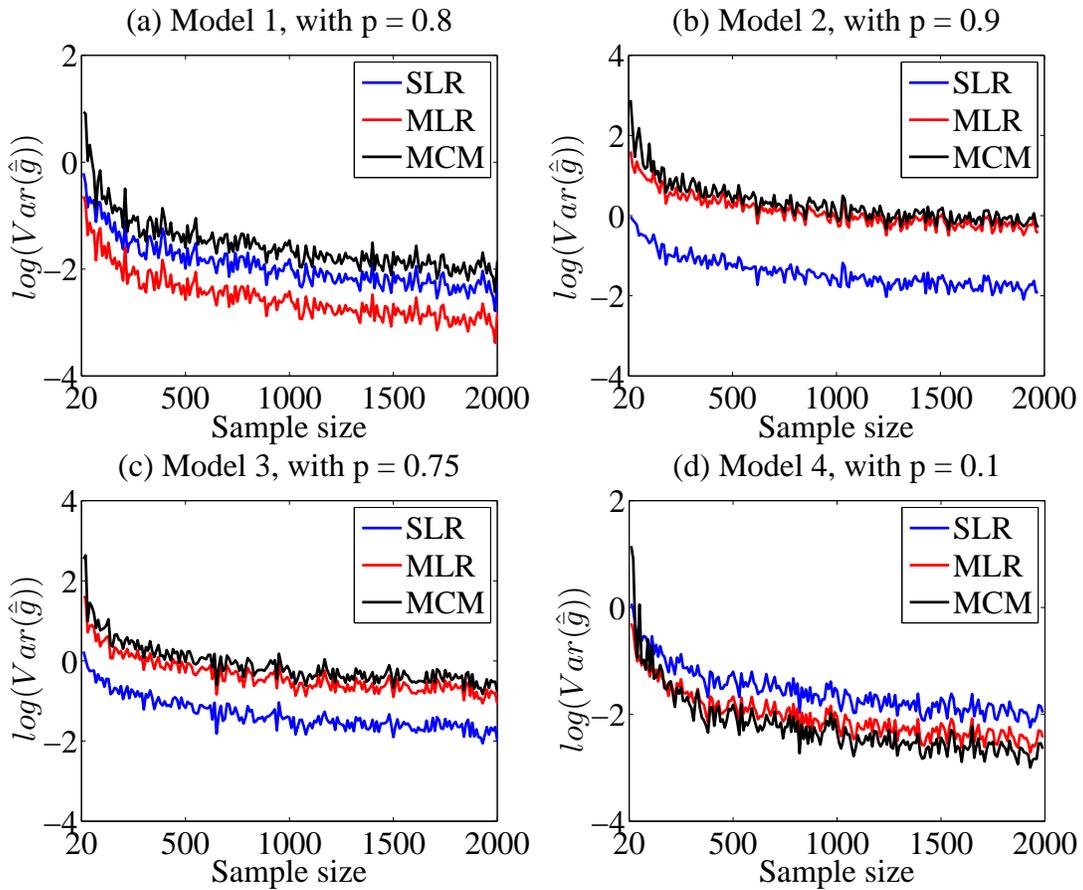}
			\caption{Variance of the three estimators of synthetic data.}
			\label{fig1}
			\end{figure}

			The results generated by nonlinear models from Fig.\ref{fig1} suggests similar claims as in Fig.\ref{figlinear}. We observe that it is not always the case that MLR yield a better estimator, although it may appear beneficial to include covariates into the model in order to improve prediction. As seen from Fig.\ref{fig1}(a), when the model is linear in the covariates, MLR has the best performance with respect to variance reduction. However, if the model is not linear then MLR does not necessarily reduce the variance of the estimator, as shown in Fig.\ref{fig1}(b) through Fig.\ref{fig1}(d). Comparing results from Fig.\ref{fig1}(b) and Fig.1(c), we observe that SLR has the lowest estimator variance in both cases, when neither $f_i$ nor $g_i$ is linear in the covariates. Thus we argue that performing SLR is the safest way to yield an estimator with least variance. What is more, if the treatment effect is linear in the covariates and the probability of treatment assignment is small, then MCM outperforms both SLR and MLR, as shown in Fig.\ref{fig1}(d). It thus serves as a compromise to use covariate information while keeping the estimator's variance low.

\end{document}